\documentclass[10pt]{amsart}
\usepackage{filecontents}

\usepackage[utf8]{inputenc} 
\usepackage{amsmath}
\usepackage{amssymb}
\usepackage{amsthm}
\usepackage{amsfonts, dsfont}
\usepackage{paralist}
\usepackage{graphics} 
\usepackage{epsfig} 
\usepackage{graphicx}  
\usepackage{epstopdf}
\usepackage{epstopdf}
\usepackage{verbatim}
\usepackage{mathrsfs}
\usepackage{mathtools}
\usepackage{pstricks}
\usepackage{cleveref}
\usepackage{relsize}
\usepackage{tikz}
\usetikzlibrary{matrix}
\usepackage{subcaption}
\usepackage{pgfplots}
\usepackage{fixltx2e}
\usepackage{enumitem}
\usepackage{ upgreek }
\usepackage{bm}
\usepackage{appendix}
\newtheorem{thm}{Theorem}[section]
\newtheorem{lemma}{Lemma}[section]
\newtheorem{corollary}{Corollary}[section]
\newtheorem{prop}{Proposition}[section]

\newtheorem{rmk}{Remark}[section]

\theoremstyle{plain}

\numberwithin{equation}{section}

\DeclareMathOperator*{\argmin}{argmin}

\newcommand{\numberset}{\mathbb}
\newcommand{\R}{\numberset{R}}

\def\eps{\varepsilon}
\def\a{\alpha}
\def\d{\delta}
\def\l{\lambda}
\def\om{\omega}

\def\de{\partial}

\def\uf{\underline{f}}

\def\M{\mathfrak{M}}

\usepackage{xcolor}

\begin{document}
\title[Eikonal equation with vanishing Lagrangian and Global Optimization]{An Eikonal equation with vanishing Lagrangian arising in Global Optimization}
\thanks{The first author is member of the Gruppo Nazionale per l'Analisi Matematica, la Probabilit\`a e le loro Applicazioni (GNAMPA) of the Istituto Nazionale di Alta Matematica (INdAM). \\
The work of the second author is funded by the Deutsche Forschungsgemeinschaft (DFG, German Research Foundation) – Projektnummer 320021702/GRK2326 – Energy, Entropy, and Dissipative Dynamics (EDDy). The results of this paper are part of his Ph.D. thesis \cite{kouhkouh22PhD} which was conducted when he was a Ph.D. student in University of Padova.
}

\author{Martino Bardi}
\address {Martino Bardi \newline \indent
{Department of Mathematics “T. Levi-Civita”, \newline \indent 
University of Padova, via Trieste, 63},
\newline \indent
{I-35121 Padova, Italy}
}
\email{\texttt{bardi@math.unipd.it}}

\author{Hicham Kouhkouh}
\address{Hicham Kouhkouh \newline \indent
{RWTH Aachen University, Institut f\"ur Mathematik,  \newline \indent 
RTG Energy, Entropy, and Dissipative Dynamics,\newline \indent
Templergraben 55 (111810)},
 \newline \indent 
{52062, Aachen, Germany}
}
\email{\texttt{kouhkouh@eddy.rwth-aachen.de}}

\date{\today}
\begin{abstract}
We show a connection between global unconstrained optimization of a continuous function $f$ and weak KAM theory for an eikonal-type equation arising also in ergodic control.
A solution $v$ of the critical Hamilton-Jacobi equation is built by a small discount approximation as well as the long time limit of an associated evolutive equation. Then $v$ is represented as the value function of a control problem with target, whose optimal trajectories are driven by a differential inclusion describing the gradient descent of $v$. Such trajectories are proved to converge to the set of minima of $f$, using tools in control theory and occupational measures. We prove also that in some cases the set of minima is reached in finite time.
\end{abstract}

\maketitle

\section{Introduction}
\label{sec: motivation}
Let $f\in C(\R^n)$ be a bounded function attaining the global minimum. Global optimization is concerned with the search of the minimum points, i.e., finding the set $\M=\argmin f$.  
For convex smooth functions this is achieved by the gradient flow, 
 i.e., by following the trajectories of $\dot{y}(s) = -\nabla f(y(s))$ from any initial point $x=y(0)$. However, if the function  $f$ is not convex the trajectory $y(\cdot)$ may converge to a local minimum or a saddle point. Several alternative algorithms have been designed 
  to handle non-convex optimization, such as the stochastic gradient descent,
simulated annealing, or consensus-based methods. 
 In particular the case of non-smooth $f$ in high dimensions is important 
 for the applications to machine learning, see, e.g., the recent paper \cite{carrillo2021consensus} and the references therein.
 
In this paper we construct and study a Lipschitz function $v : \R^n \to \R$ such that the following normalized non-smooth gradient descent differential inclusion
\begin{equation}
\label{grad_desc}
 \dot{y}(s)\in \left\{ - \frac{p}{|p|}\,,\; p\in D^{-}v(y(s))\right\},\; \text{ for a.e. }\, s>0 ,
\end{equation}
has a solution for any initial condition $x=y(0)$ and all solutions converge to $\M$ as $t\to +\infty$. Here $D^{-}v$ is the sub-differential of the theory of viscosity solutions (see, e.g., \cite{bardi2008optimal}). The construction of such a generating function $v$ is based on a classical problem for Hamilton-Jacobi equations: find a constant $c$ such that the stationary equation
\begin{equation}
\label{crit_eq}
H(x, Dv)=c 
\quad \text{in } \R^n 
\end{equation}
has a solution $v$. 
The minimal $c$ with this property is the critical value of the Hamiltonian $H$ and, if $H(x,\cdot)$ is convex, it is also the value of an optimal control problem with ergodic cost 
 having $H$ as its Bellman Hamiltonian. If the critical solution $v$ is interpreted in the viscosity sense, the problem fits in the weak KAM theory, and it is well-known that, for $H=\frac{1}{2}|p|^2 - f(x)$  with  $f$ periodic, $c=-\min f$ \cite{lions1987homogenization, fathi2008weak}; moreover the same holds for any bounded $f\in C^2(\R^n)$ by a result of Fathi and Maderna \cite{fathi2007weak},  { and for uniformly continuous $f$ as proved by Barles and Roquejoffre \cite{barles2006ergodic}. In Section \ref{sec: asymptotics} we extend such result to 
  $f\in C(\R^n)$, 
  bounded, and attaining its minimum}. We also prove that $\min f$ and $v$ solving the critical equation
\begin{equation*}
   \min f 
    + \frac{1}{2}|\nabla v(x)|^{2} = f(x) \quad \text{ in } \mathds{R}^{n}
\end{equation*}
can be approximated in two ways: by the solution of the stationary equation
\begin{equation}
\label{eq: u_l}
    \lambda u_{\lambda} + \frac{1}{2}|Du_{\lambda}|^{2} = f(x),\quad x\in\mathds{R}^{n},
\end{equation}
as $\l\to 0+$, the so-called small discount limit, as well as by the long-time limit of the solution of the evolution equation
\begin{equation}
\label{eq: evol}
\partial_{t}u
  + \frac{1}{2}|Du
  |^{2} =\, f(x),\; \text{ in }\;\mathds{R}^n \times (0,+\infty), \quad
    u(x,0) = 0. 
\end{equation}
More precisely, for the evolutive equation \eqref{eq: evol} we prove
\begin{equation}
\label{longtime}
\lim_{t\to+\infty}\left(u(x,t) -  t \min f\right)  = v(x) \quad \text{locally uniformly in } \mathds{R}^{n}.
\end{equation}
Note that the two problems \eqref{eq: u_l} and \eqref{eq: evol} do not require the a-priori knowledge of $\min f$ and $\argmin f$.  { If, in addition, $f$ is  Lipschitz and semiconcave, we show that $v$ is semiconcave and } $Du_\l$ and $D_xu(\cdot,t)$ both converge (a.e.) to $Dv$, therefore giving an approximation of the gradient descent equation \eqref{grad_desc}.
Moreover, in this case  \eqref{grad_desc} becomes the classical normalised gradient descent
\[
  \dot{y}(t)= - \frac{Dv(y(t))}{|Dv(y(t))|},\quad \forall \,t>0 .
\]


The main result of the paper is the convergence of the gradient descent trajectories \eqref{grad_desc} to the set $\M$ of minima of $f$. This is done in Section \ref{sec: prob target} after observing that $v$ solves also the Dirichlet problem for the eikonal equation
\begin{equation}
\label{eq: motivation - Dirichlet prob}
\left\{\quad
\begin{aligned}
     |\nabla v(x)| &= \ell(x), & x\in \mathds{R}^{n}\setminus \mathfrak{M}\\
     v(x) &= 0, & x\in \mathfrak{M}
\end{aligned}\right.
\end{equation}
with $\ell(x) := \sqrt{2(f(x) - \min f)}$. 
 (In fact, our analysis of this problem requires only that $\ell\in C(\R^n)$ is bounded, non-negative, and $\M=\{x : \ell(x)=0\}$).
We exploit that the unique solution of \eqref{eq: motivation - Dirichlet prob} is the value function
\begin{equation*}
    v(x)=\inf\limits_{\alpha({\cdot})
    }\int_{0}^{t_{x}(\alpha)} \ell(y_{x}^{\alpha}(s))\,\text{d}s , \quad \dot{y}^{\alpha}_{x}(s) = \alpha(s) ,\,
    \text{ for  } s > 0,\quad y_{x}^{\alpha}(0)=x,
\end{equation*}
where $\a$ is measurable, $|\alpha(s)|\leq 1$, and $t_{x}(\alpha)$ is the first time the trajectory $y_{x}^{\alpha}$ hits $\M$. 
We show that optimal trajectories exist, satisfy the gradient descent inclusion \eqref{grad_desc}, and tend to $\M$ as $t\to +\infty$ under a slightly strengthened positivity condition { at infinity for }$\ell$. A crucial new tool for the proof are the occupational measures associated to these { trajectories}.

{ In the final section of the paper we give sufficient conditions such that 
the optimal trajectories reach $\M$ in finite time. This is a nontrivial problem even when $v$ is smooth, because it is equivalent to the finite length of gradient orbits $\dot z(s)=-Dv(z(s))$, a question with a very large literature and open problems, see, e.g., \cite{bolte2007lojasiewicz, daniilidis2010asymptotic} and the references therein. Here we prove the finite hitting time by assuming a bound from below on $\ell$ near the target and showing an inequality of \L ojasiewicz type along optimal trajectories.}

In a forthcoming companion paper we also study the 
 approximation of $v$ and $\M$ by vanishing viscosity. We add to \eqref{eq: u_l} a term $-\eps \Delta u_\l$ and let $\l\to 0+$ to get the viscous critical equation
 \begin{equation*}
U^\eps - \eps \Delta v^\eps(x) 
    + \frac{1}{2}|\nabla v^\eps(x)|^{2} = f(x) \quad \text{ in } \mathds{R}^{n} ,
\end{equation*}
where $U^\eps$ is a constant. We prove that $0\leq U^\eps - \min f \leq C\eps^\beta$ for some $\beta>0$. Then we define the approximate stochastic gradient descent
\begin{equation*}
    \text{d}X
    _{s} = -\nabla u
    _{\lambda}(X
    _{s})\,\text{d}s + \sqrt{2\varepsilon}\,\text{d}W_{s},
\end{equation*}
and show that the trajectories converge 
to $\M$ in a suitable sense, for small $\l$ and $\eps$. These results can be found also in the second author's thesis \cite{kouhkouh22PhD}.

Note that \eqref{eq: evol} is the classical Hamilton-Jacobi equation with the mechanical Hamiltonian $H(x,p)=\frac 12 |p|^2-f(x)$, where $-f$ is the potential energy. Then our results of Section \ref{sec: asymptotics} have an  interpretation in analytical mechanics. For instance, the long-time behavior \eqref{longtime} describes 
 a thermodynamical trend to equilibrium in a non-turbulent gas or fluid: see  \cite{cardin2008fluid, cardin2015elementary}.

We do not attempt to review all the literature related to the topics mentioned above. 
For weak KAM theory 
 on compact manifolds we refer to \cite{fathi1997theoreme, fathi2008weak, evans2004survey}, and for the PDE approach to ergodic control, mostly under periodicity assumptions, 
 the reader can consult \cite{arisawa1998ergodicsto, alvarez2010ergodicity} and the references therein. When the state space is not bounded one must add 
 conditions to get some compactness. In addition to \cite{fathi2007weak, barles2006ergodic} already quoted, such problems were studied in all $\R^n$ by \cite{artstein2000value, motta2015asymptotic, nguyen2016singularly, cannarsa2020asymptotic, cannarsa2022asymptotic, ishii2020vanishing} assuming that $f$ is large enough at infinity, and by \cite{fujita2006asymptotic, ishii2008asymptotic, kaise2009ergodic} for equations involving a linear first order term that satisfies a recurrence condition, { see also the references therein}. Here, instead,  we get compactness from the { boundedness} 
  of $f$ { and the assumption that its minimum is attained}. Several of the results just quoted were used for homogenisation and singular perturbation problems, e.g., \cite{lions1987homogenization, artstein2000value, alvarez2010ergodicity, nguyen2016singularly}, so we believe that also our results will have such applications.
 
 The Dirichlet problem \eqref{eq: motivation - Dirichlet prob} with $\ell$ 
  vanishing at the boundary was studied, e.g., in \cite{soravia1999optimality, malisoff2004bounded, motta2015value}. The case of a cost that does not vanish is part of time-optimal control and it is treated 
  in \cite{bardi2008optimal}, see also the references therein. The synthesis of an optimal  feedback from the value function $v$ leading to \eqref{grad_desc} uses method from \cite{bardi2008optimal} based on the earlier papers \cite{berkovitz1989optimal, frankowska1989optimal}.

We do not try here to design algorithms for global optimization based on the previous results. Let us mention, however, that { an } 
 efficient numerical method for computing at the same time $c$ and $v$ in the critical/ergodic PDE \eqref{crit_eq} was proposed in  \cite{cacace2016generalized}.

The paper is organized as follows. In 
Section \ref{sec: stationary} we prove the weak KAM theorem by the small discount approximation \eqref{eq: u_l} and in Section \ref{sec: evolutive} we study the long-time asymptotics of solutions to \eqref{eq: evol}. Section \ref{sec: prob target} is devoted to the optimal control problem with target $\M$ associated to \eqref{eq: motivation - Dirichlet prob} and Section \ref{sec:  diff inclusion}  to deriving the gradient descent inclusion \eqref{grad_desc} for the optimal trajectories. In Section \ref{sec:  conv traj} we prove that such trajectories converge to $\M$, 
and in Section \ref{sec: finite time} we show { two cases} where the hitting time is finite.

\section{A weak KAM theorem and approximation of the critical solution} 
\label{sec: asymptotics}

We introduce the following assumptions and refer to them wherever it is needed:
\textit{Assumptions \textbf{(A)}}
\begin{enumerate}[label=\textbf{(A\arabic*)}]
\item $f : \R ^n\to\R$ is 
     continuous and
    \begin{equation}
    \label{bounds f}
    \exists\;\underline{f},\,\overline{f}\; \text{s.t. }\; \underline{f} \leq f(x) \leq \overline{f},\quad \forall\;x\in \mathds{R}^n,
    \end{equation}
\item $f$ attains the minimum, i.e.,  
  \begin{equation}\label{minimizers}
    \mathfrak{M}:=\{x\in\mathds{R}^{n}\,:\, f(x) = \underline{f}:= \min\limits_{z\in\mathds{R}^{n}}f(z)\} \ne \emptyset 
\end{equation}
\end{enumerate}
\textit{Assumptions \textbf{(B)}}
\begin{enumerate}[label=\textbf{(B\arabic*)}]
\item $f$ is $C_{1}$-Lipschitz continuous, i.e. $C_{1}= 
\|\nabla f\|_\infty$,
    \item $f$ is $C_{2}$-semiconcave, i.e., $D^{2}_{\xi\xi}f \leq C_{2}$ a.e. for all $\xi\in\mathds{R}^{n}$ s.t. $|\xi|=1$, where $D^{2}_{\xi\xi}f$ is the second order derivative of $f$ in the direction $\xi$.
\end{enumerate}

A weak KAM theorem for the Hamiltonian $H(x,p)=\frac{1}{2}|p|^2 - f(x)$ should give conditions under which there exists a constant $U\in\R$, the (Man\'e) critical value, such that the equation
\begin{equation}\label{critical}
   U 
    + \frac{1}{2}|\nabla v(x)|^{2} = f(x),\quad \text{ in } \mathds{R}^{n}.
\end{equation}
has a viscosity solution $v$. Clearly any critical value must satisfy $U\leq \uf$. In this section we prove under the current assumptions that $\uf$ is a critical value and construct the solution $v$ by two different approximation procedures, both 
having an interpretation in terms of ergodic problems in optimal control. 

The fact that $\uf$ is the maximal critical value was proved in \cite{fathi2007weak} for $f\in C^{2}$ 
and with $\R^n$ replaced by any complete Riemannian manifold, by methods of weak KAM theory different form ours.


\subsection{The small discount limit}
\label{sec: stationary}
We consider the stationary approximation of \eqref{critical}
\begin{equation}
\label{eq: u lambda - 1 order}
    \lambda u_{\lambda} + \frac{1}{2}|Du_{\lambda}|^{2} = f(x),\quad x\in\mathds{R}^{n},
\end{equation}
where $\l >0$ will be sent to $0$. The viscosity solution $u_\l$ is known to be the value function 
of the following infinite horizon discounted optimal control problem 
\begin{equation}\label{eq: u lambda}
\begin{aligned}
    u_{\lambda}(x) = \inf\limits_{\alpha_{\cdot}} \;
    & 
    J(x,\alpha_{\cdot}):=\int_{0}^{+\infty}\left(\frac{1}{2}|\alpha_{t}|^{2}+f(x(t))\right)e^{-\lambda t}\,\text{d}t,\\
    & \text{s.t. }\; \dot{x}(s) = \alpha_{s},\quad x(0)=x\in\mathds{R}^n,\quad s\geq 0
\end{aligned}
\end{equation}
where the controls $\alpha.:[0,+\infty)\to \mathds{R}^{n}$ are measurable function 
(see, e.g., 
\cite[Chapter III]{bardi2008optimal}).    
The main result of this section is the following.
\begin{thm}\label{thm: ergodic 2}
Under assumptions (A), as $\lambda\to 0$
\begin{equation*} 
    \lambda u_{\lambda}(x) \to \underline{f} \quad and \quad
             u_{\lambda}(x) -  \uf\l^{-1}  \to v(x) \quad \text{locally uniformly in } \mathds{R}^{n} ,
\end{equation*}
where $v(\cdot)$ is a Lipschitz continuous viscosity solution to
\begin{equation}
\label{eq: ergodic pde}
    \underline{f} + \frac{1}{2}|Dv(x)|^{2} = f(x),\quad x\in \mathds{R}^n .
\end{equation}
Moreover $v\geq 0$ in $\R^n$ and null on $\M$, and it is the unique viscosity solution of \eqref{eq: ergodic pde} in $\R^n\setminus\M$ vanishing on $\de\M$ and bounded from below.\\
If we assume moreover that assumptions (B) hold, then
\begin{equation*} 
 Du_{\lambda}(x)\to Dv(x)  \quad a.e. 
\end{equation*}
\end{thm}
For the proof we need some estimates uniform in $\l$. The first Lemma is known and we omit the proof (see \cite{kouhkouh22PhD} for the details).

\begin{lemma}
\label{lem: bound lambda u}
Under the assumption (A1),  for all $x\in \mathds{R}^n$ and $\lambda>0$,
\begin{equation}
\label{ufof}
    \underline{f}\;\leq\; \lambda u_{\lambda}(x)\;\leq \; \overline{f}
\end{equation}
\begin{equation}
\label{eq: unif grad est _ proof - stat}
    |Du_{\lambda}(x)| \leq \sqrt{4\|f\|_{\infty}} \quad \text{a.e.} .
\end{equation}
\end{lemma}

\begin{lemma}\label{lem: semiconcavity - lambda}
Assume (A) and (B) hold. Then $u_{\lambda}$ is $\widetilde{C}_{3}-$semiconcave, where $\widetilde{C}_{3}$ is a positive constant independent of $\lambda> 0$.
\end{lemma}

\begin{proof} We will skip the more standard parts and refer to
\cite{kouhkouh22PhD} for the complete details.
We use the vanishing viscosity approximation
\begin{equation}
\label{eq: pde u lambda}
    \lambda u^{\varepsilon}_{\lambda} - \varepsilon \Delta u^{\varepsilon}_{\lambda} + \frac{1}{2}|Du^{\varepsilon}_{\lambda}|^{2} = f(x),\quad x\in\mathds{R}^{n}.
\end{equation}
%
%
We fix $\xi\in\mathds{R}^{n}$ such that $|\xi| = 1$ and denote $\omega_{\lambda} (x) := D^{2}_{_{\xi\xi}}u^{\varepsilon}_{\lambda}(x)$  the second order derivative in the direction $\xi$. 
The estimates $\omega_{\lambda}(x)\leq \lambda^{-1}C_{2}$ and 
\begin{equation}
\label{eq: Du lambda}
    |Du^{\varepsilon}_{\lambda}(x)|\leq \lambda^{-1}C_{1} 
\end{equation}
are standard and can be got, for instance, by representing $u^{\varepsilon}_{\lambda}$ as the value function of the stochastic infinite-horizon discounted optimal control problem associated to \eqref{eq: pde u lambda} and exploiting the $C_2$-semiconcavity and $C_1$-Lipschitz continuity of $f$.

Next we differentiate twice \eqref{eq: pde u lambda} in the direction of $\xi$ and obtain 
\begin{equation*}
    - \varepsilon\Delta \omega_{\lambda} + Du^{\varepsilon}_{\lambda}\cdot D\omega_{\lambda} + |D_{_{\xi}}Du^{\varepsilon}_{\lambda}|^{2} + \lambda \omega_{\lambda} = D^{2}_{_{\xi\xi}}f,\quad \text{in } \mathds{R}^n.
\end{equation*}
By $\omega_{\lambda}^{2}\leq |D_{_{\xi}}D u_{\lambda}|^{2}$ and  the semiconcavity assumption  $D^{2}_{_{\xi\xi}}f \leq C_{2}$ we get 
\begin{equation}
\label{eq: om_pf_lambda}
    - \varepsilon\Delta \omega_{\lambda} + Du^{\varepsilon}_{\lambda}\cdot D\omega_{\lambda} + \omega_{\lambda}^{2} + \lambda \omega_{\lambda} \leq  C_{2},\quad \text{in } \mathds{R}^n.
\end{equation}
In the case $\om_\l$ attains its maximum at some $\bar x$ we have
\[
\om_\l^2(\bar x) + \l\om_\l(\bar x) \leq C_2 .
\]
By the elementary inequality $\frac{1}{2}\left( z^{2} - \lambda^{2}\right) \leq z^{2} + \lambda z$ we get, for $\l\leq 1$,
\[
\om_\l^2(\bar x) \leq 2C_2 +1
\]
and then we easily reach the conclusion. For the general case we set, for $\beta>0$ to be chosen,
\begin{equation*}
 \Psi_{\lambda}(x):= \omega_{\lambda}(x) - \beta \log(1+|x|^{2}) .
\end{equation*}
Since $\omega_{\lambda}$ is bounded from above, 
$\Psi_{\lambda}$ attains a global maximum in $\mathds{R}^n$, say at $\overline{x}$ 
(which depends on $\lambda$ and $\beta$). By evaluating \eqref{eq: om_pf_lambda} in $\overline{x}$, 
after some calculations and using the bound 
\eqref{eq: Du lambda} we arrive at 
\begin{equation*}
    \omega_{\lambda}^{2}(\overline{x}) + \lambda\omega_{\lambda}(\overline{x}) \leq C_{2} + 2\varepsilon\beta n + 2\beta \lambda^{-1} C_{1}
\end{equation*}
Arguing as above we get, for $\beta\leq\l/2\leq 1$,
\begin{equation}
\label{eq: om_lambda_C3}
    \omega_{\lambda}(\overline{x})^{2} \leq 2(C_{1} + C_{2} + 2\varepsilon n) + 1 .
\end{equation}
Now we 
claim that
\begin{equation*}
    \omega_{\lambda}(x) \leq C_{3} \coloneqq \sqrt{2(C_{1} + C_{2} + 2\varepsilon\,n) + 1},\quad \text{for all }\, x\in \mathds{R}^n.
\end{equation*}
To prove the claim we suppose by contradiction there exists $y\in \mathds{R}^{n}$ such that $\omega_{\lambda}(y,s) - C_{3}=: \d >0$.
Denote $g(x) := \log(1+|x|^{2})$ and choose $\beta>0$ small enough such that 
$\beta g(y)\leq \frac{\delta}{2}$. 
Then 
\begin{equation*}
    0< \frac{\delta}{2}\leq \delta-\beta g(y) = \omega_{\lambda}(y) - \beta g(y) - C_{3} = \Psi_{\lambda}(y) - C_{3}
\end{equation*}
and hence $\Psi_{\lambda}(\overline{x}) - C_{3} >0$. 
On the other hand \eqref{eq: om_lambda_C3} gives
  $ \omega_{\lambda}(\overline{x})\leq C_{3}$
and 
\begin{equation*}
    \Psi_{\lambda}(\overline{x}) - C_{3} \leq -\beta g(\overline{x}) \leq 0 ,
\end{equation*}
which is the desired  contradiction. 
This proves the claim and the $C_3$-semiconcavity of $u^\eps_{\lambda}$, uniformly in $\lambda$, 
for every $0<\varepsilon\leq 1$. 
Finally we let $\eps \to 0$ in \eqref{eq: pde u lambda}  and get that the solution $u_{\lambda}$ to \eqref{eq: u lambda - 1 order} is semi-concave with constant $\widetilde{C}_{3}\coloneqq\sqrt{2(C_{1} + C_{2}) + 1}$.
\end{proof}

%
%

\begin{proof}{(Theorem \ref{thm: ergodic 2})}
First we claim 
 that $\lambda u_{\lambda}(\bar{x})=\underline{f}$  if $\bar x\in \M$ (i.e., $f(\bar
{x})=\underline{f}= \min f$), for all $\lambda>0$. 
In fact, for such $\bar x$,
 \begin{equation*}
    u_{\lambda}(\overline{x}) = \inf\limits_{\alpha_{\cdot}}\int_{0}^{+\infty}\left(\frac{1}{2}|\alpha_{t}|^{2} + f(x(t))\right)e^{-\lambda\,t} \, \text{d}t \leq \int_{0}^{+\infty}f(\overline{x})e^{-\lambda\,t}\,\text{d}t = \uf\l^{-1} ,
\end{equation*}
where the 
 inequality follows from the choice $\alpha_{\cdot}\equiv 0$. 
 The other inequality $\geq$ is true for all $x\in\mathds{R}^{n}$ by Lemma \ref{lem: bound lambda u}, so the claim is proved.

Now we denote $R\coloneqq \sqrt{4\|f\|_{\infty}}$ and use the gradient bound  \eqref{eq: unif grad est _ proof - stat} to get 
\[
|\l u_{\lambda} (x)-\uf|\leq \l R \,\text{dist}(x, \M)  \quad \forall \, x\in\R^n . 
\]
Then
$\lambda u_{\lambda}(x) \to \underline{f}$ 
{locally uniformly}.

Define 
 $\varphi_{\lambda}(\cdot) := u_\lambda(\cdot) - \uf \l^{-1}\geq 0$ 
 and use \eqref{eq: unif grad est _ proof - stat} to get, 
  for all $x,y\in\mathds{R}^n$,
\begin{equation}
\label{eq: ineq proof thm ergodic 2}
        |\varphi_{\lambda}(x)|  \leq R \,\text{dist}(x, \M) , 
        \quad  \quad 
        |\varphi_{\lambda}(x)-\varphi_{\lambda}(y)|  \leq R\,|x-y|.
\end{equation}
Hence, $\{\varphi_{\lambda}(\cdot)\}_{\lambda\in(0,1)}$ is a uniformly bounded and equi-continuous family on any ball of $\mathds{R}^n$. So we can choose a sequence $\lambda_{k}\to 0$ 
as $k\to +\infty$, 
such that $\varphi_{\lambda_{k}}(\cdot)\to v(\cdot)\in C(\mathds{R}^n)
$ locally uniformly. 
Plugging $\varphi_{\lambda}$ in     \eqref{eq: u lambda - 1 order} we get 
\begin{equation*}
  \l\varphi_{\lambda} + \uf + \frac{1}{2}|D\varphi_{\lambda}(x)|^{2} = f(x),\quad x\in \mathds{R}^n .
\end{equation*}
We let $\lambda_{k}\to 0$ and use the stability of  viscosity solutions to find that $v$ satisfies \eqref{eq: ergodic pde}. 

Now we note that \eqref{eq: ergodic pde} is an eikonal equation with right hand side $f(x) - \uf > 0$ in $\R^n \setminus \M$, $v\geq 0$ and $v=0$ on $\de\M$. This Dirichlet boundary value problem is known to have a unique viscosity solution bounded from below. Therefore the convergence of $\varphi_{\lambda}$ is for $\l\to 0$ and not only on subsequences.

The convergence of the gradient $Du_{\lambda
}(\cdot)$ to $Dv(\cdot)$ is a direct consequence of \cite[Theorem  3.3.3]{cannarsa2004semiconcave}, recalling that $|\varphi_{\lambda
}(x)|\leq R\,|x|$ and using the uniform semiconcavity estimate in Lemma \ref{lem: semiconcavity - lambda}.
\end{proof}

\subsection{Long time asymptotics} 
\label{sec: evolutive}

Here we consider the evolutive Hamilton-Jacobi equation
\begin{equation}
\label{eq: Cauchy pb- 1st order}
\left\{
\begin{aligned}
    \partial_{t}u(x,t)  + \frac{1}{2}|Du(x,t)|^{2} =&\, f(x),\quad && (x,t)\in\mathds{R}^n \times (0,+\infty) ,\\
    u(x,0) =&\, 0,\quad && x\in\mathds{R}^n ,
\end{aligned}\right.
\end{equation}
where $D=\nabla=D_x$ denotes the gradient with respect to the space variables $x$, and we will study the limit as $t\to +\infty$. The viscosity solution $u(x,t)$ is known to be the value function of 
the following finite-horizon optimal control problem
\begin{equation}\label{eq: u(x,t)}
\begin{aligned}
    u(x,t) = \inf\limits_{\alpha_{\cdot}} \; & J(x,t,\alpha_{\cdot}):=\int_{0}^{t}\; \frac{1}{2}|\alpha_{s}|^{2} + f(x(s))\;\text{d}s,\\
    & \text{s.t. }\; \dot{x}(s) = \alpha_{s},\quad x(0)=x\in\mathds{R}^n
\end{aligned}
\end{equation}
where $\alpha.:[0,+\infty)\to \mathds{R}^{n}$ are measurable functions
(see e.g. \cite[Chapter II]{fleming2006controlled} or \cite[Chapter III]{bardi2008optimal}).
The main result of this section is the following.

\begin{thm}\label{thm: ergodic 1}
Under assumptions (A), as $t\to+\infty$,
\begin{equation*} 
    \frac{u(x,t)}t  \to \underline{f} \quad and \quad
     u(x,t) -  \uf t  \to v(x) \quad \text{locally uniformly in } \mathds{R}^{n} ,
\end{equation*}
where $v(\cdot)$ is the viscosity solution of  \eqref{eq: ergodic pde} found in Theorem \ref{thm: ergodic 2}.\\
If we assume moreover that assumptions (B) hold, then
\begin{equation*} 
 D_x u(x,t) \to Dv(x)  \quad a.e. 
\end{equation*}
\end{thm}

To proceed with its proof we need some estimates uniform in $t$.

\begin{lemma}
\label{lem: bound average unif x} Under the assumption (A1),  for all $(x,t)\in\mathds{R}^n \times (0,+\infty)$, 
\begin{equation}
    \underline{f}\leq \frac{u(x,t)}{t} \leq \overline{f} ,
\end{equation}
\begin{equation}
  \label{bound_de_t}
    \left|\partial_{t}u(x,t)\right| \leq \|f\|_{\infty}  \quad \text{a.e.} ,
\end{equation}
\begin{equation}
    \label{eq: unif grad est _ proof}
    |Du(x,t)|\leq \sqrt{4\|f\|_{\infty}}  \quad \text{a.e.} .
\end{equation}
\end{lemma}
\begin{proof} The arguments are standard, for the reader's convenience we show \eqref{bound_de_t}.
Fix $h\in \mathds{R}$ and $x\in\mathds{R}^n$. 
Note first that $
|u(x,h)|\leq |h|\|f\|_{\infty}$.
Let us now denote $\overline{v}(x,t) := u(x,t+h) + |h|\|f\|_{\infty}$. Both $u$ and $\overline{v}$ solve the same PDE in \eqref{eq: Cauchy pb- 1st order} with initial conditions $u(x,0) = 0$ and $\overline{v}(x,0) = u(x,h) + |h|\|f\|_{\infty}
\geq 0$, hence by the comparison principle in \cite[Theorem  2.1]{da2006uniqueness} 
we get $u(x,t)\leq \overline{v}(x,t)$. 

Conversely, 
$\underline{v}(x,t) := u(x,t+h) - |h|\|f\|_{\infty}$ 
solves the same PDE in \eqref{eq: Cauchy pb- 1st order} with initial condition $\underline{v}(x,0) = u(x,h) - |h|\|f\|_{\infty} \leq u(x,0) = 0$. The same comparison principle now implies that $\underline{v}(x,t)\leq u(x,t)$. Therefore, one gets $|u(x,t+h)-u(x,t)|\leq |h|\|f\|_{\infty}$.
\end{proof}



\begin{lemma}\label{lem: semiconcavity}
Assume (A) and (B) hold. Then $u$ is $\widetilde{C}_{3}-$semiconcave, where $\widetilde{C}_{3}$ is a positive constant independent of $t\geq 0$.
\end{lemma}

\begin{proof}
As we did in the proof of Lemma \ref{lem: semiconcavity - lambda}, we consider the vanishing viscosity approximation
\begin{equation}\label{eq: hjb T}
    \left\{
    \begin{aligned}
        &\partial_{t}u^{\varepsilon} - \varepsilon\Delta u^{\varepsilon} + \frac{1}{2}|\nabla u^{\varepsilon}|^{2} = f(x),\quad (x,t)\in \mathds{R}^{n}\times (0, \infty)\\  
        &u^{\varepsilon}(x,0) = 0,\quad x\in\mathds{R}^n
    \end{aligned}
    \right.
\end{equation}
It is known that $u^\eps$ is the value function of the stochastic control problem
\begin{equation}
    \label{eq: value T}
    u^{\varepsilon}(x,t) = \inf\limits_{\alpha_{\cdot}\in\mathcal{A}} 
    \mathds{E}\left[\int_{0}^{t}\frac{1}{2}|\alpha_{s}|^{2} + f(X_{s})\,\text{d}s\,\bigg|\, X_{0}=x\right], \quad
    \text{d}X_{s} = \alpha_{s}\,\text{d}s + \sqrt{2\varepsilon}\,\text{d}W_{s} . 
\end{equation}

Take 
$\xi\in\mathds{R}^{n}$ with $|\xi| = 1$ and let $\omega(x,t) := D^{2}_{\xi\xi}u^{\varepsilon}(x,t)$ be the second order derivative in space in the direction $\xi$. 
We claim first that $\omega(x,t)\leq t\,C_{2}$ 
or, equivalently, the value function $u^{\varepsilon}(x,t)$ is $t\,C_{2}$-semiconcave in the spatial variable $x$. Let $\delta>0$ and take a $\frac{\delta}{2}$-optimal control for the initial point $x$. By using  the same control for the initial points $x+h$ and $x-h$ we get 
\begin{equation}
\begin{aligned}
    &u^{\varepsilon}(x+h,t) - 2 u^{\varepsilon}(x,t) + u^{\varepsilon}(x-h,t) - \delta \leq\\
    &\quad\quad\quad\quad\quad\quad \mathds{E}\left[\int_{0}^{t} f(X^{x+h}_{s}) - 2 f(X^{x}_{s}) + f(X^{x-h}_{s})\; \text{d}s\right]
\end{aligned}
\end{equation}
From the controled diffusion in \eqref{eq: value T} we have $X^{x}_{s} = \frac{1}{2}\left(X^{x+h}_{s} + X^{x-h}_{s}\right)$, and $f$  $C_{2}$-semiconcave implies 
\begin{equation}
\begin{aligned}
    &\mathds{E}\left[\int_{0}^{t} f(X^{x+h}_{s}) - 2 f(X^{x}_{s}) + f(X^{x-h}_{s}) \; \text{d}s\right] \\
    &\quad\quad\quad\quad \leq\; C_{2} \mathds{E}\left[\int_{0}^{t} \frac{1}{4}\left|X^{x+h}_{s} - X^{x-h}_{s}\right|^{2} \; \text{d}s\right] 
    \leq \; t\,C_{2} \,|h|^{2}
\end{aligned}
\end{equation}
Since $\delta>0$ is arbitrary we have proved the claim. 
Similar computations (see \cite{kouhkouh22PhD}) yield
\begin{equation}\label{eq: lip estimate x}
    |Du^{\varepsilon}(x,t)| \leq t \,C_{1}.
\end{equation}
          
            Next we differentiate twice \eqref{eq: hjb T} in the direction of $\xi$ and obtain
            \begin{equation}
    \partial_{t}\omega - \varepsilon\Delta \omega + Du^{\varepsilon}\cdot D\omega + |D_{\xi}Du^{\varepsilon}|^{2} = D_{\xi\xi}f,\quad \text{in } \mathds{R}^n\times (0,T].
\end{equation}
Since $\omega^{2}\leq |D_{\xi}D u^{\varepsilon}|^{2}$ and by the semiconcavity assumption  $D^{2}_{\xi\xi}f \leq C_{2}$ 
\begin{equation}
\label{eq: omega _ proof T}
    \partial_{t}\omega - \varepsilon\Delta \omega + Du^{\varepsilon}\cdot D\omega + \omega^{2} \leq  C_{2},\quad \text{in } \mathds{R}^n\times (0,+\infty).
\end{equation}
Now set $g(x) := \log(1+|x|^{2})$ and $\Phi(x,t):= \omega(x,t) - \beta g(x)$, in $\mathds{R}^n\times (0,+\infty)$ for some $\beta>0$ to be made precise. 
Since $\omega$ is bounded from above 
for $0\leq t\leq T$, 
$\Phi$ admits a global maximum in $\mathds{R}^n \times [0,T]$. Let $(\overline{x},\overline{t})$ 
be such a maximum point. 
We consider first the case $\overline{t}\in (0,T)$ and evaluate \eqref{eq: omega _ proof T} in $(\overline{x},\overline{t})$ to get
\begin{equation}
\label{eq: omega _ proof 3 - T lambda}
    \omega^{2}(\overline{x},\overline{t}) 
      \leq C_{2} + 2 \varepsilon\beta \frac{n+(n-2)|\overline{x}|^{2}}{(1+|\overline{x}|^{2})^{2}}-2\beta Du^{\varepsilon}(\overline{x},\overline{t})\cdot \frac{\overline{x}}{1+|\overline{x}|^{2}}
\end{equation}
Note that $x\in \mathds{R}^n\mapsto \frac{n+(n-2)|x|^{2}}{(1+|x|^{2})^{2}}$ has a global maximum in $x=0$ for $n\geq 2$, and $\frac{x}{1+|x|^{2}}$ is bounded. 
Then, by \eqref{eq: lip estimate x} 
the bound in \eqref{eq: omega _ proof 3 - T lambda} gives
\begin{equation*}
    \omega^{2}(\overline{x},\overline{t}) \leq C_{2} + 2\varepsilon\beta n + 2\beta \,T\, C_{1} .
\end{equation*}
We choose $\beta
$ and $T$ such that 
$\beta\leq 1/(2T)<1$. Then
\begin{equation}
\label{eq: omega _ proof 4 T lambda}
    \omega(\overline{x},\overline{t})^{2}  \leq 
     C_{2} + C_{1} + 2n\varepsilon .
\end{equation}

 On the other hand, if $\overline{t}= 0$, 
 $u_{\lambda}(x,0)=0$ for all $x$ implies $\omega(\overline{x},0) = 0$ and \eqref{eq: omega _ proof 4 T lambda} still holds. 
 And if $\overline{t}= T$ 
 then $\partial_{t}\Phi(\overline{x},T) \geq 0$, i.e., $\partial_{t}\omega(\overline{x},T)\geq 0$ and \eqref{eq: omega _ proof 4 T lambda} still holds.  Therefore we have
\begin{equation}
\label{eq: omega C 3 - T lambda}
    \omega(\overline{x},\overline{t})
    \leq C_{3}:=\sqrt{C_{1} + C_{2} + 2\varepsilon n} .
\end{equation}

We are now ready to prove  that $\omega(x,t) \leq C_{3}$ for all $(x,t)\in \mathds{R}^n\times (0,+\infty)$. 
As in the proof of Theorem \ref{thm: ergodic 2} we suppose by contradiction there exists $(y,s)$ 
such $\omega(y,s) - C_{3}=: \delta >0$. Without loss of generality, we can choose $T>0$ large enough such that $s< T$.
Then we argue exactly as in the proof of Theorem \ref{thm: ergodic 2}
and reach a contradiction by choosing $\beta$ 
such that 
$\beta g(y)\leq \frac{\delta}{2}$. 
This proves the $C_3$-semiconcavity of $u$ with respect to $x$ uniformly in $t$, 
for every 
 $0<\varepsilon\leq 1$.
Finally, we let $\eps\to 0$ in \eqref{eq: hjb T} and get that the solution $u$ to \eqref{eq: Cauchy pb- 1st order} is  semi-concave in $x$ with constant $\widetilde{C}_{3}\coloneqq\sqrt{C_{1} + C_{2}}$. 
\end{proof}


\begin{proof}{(Theorem \ref{thm: ergodic 1})}\\
First we observe
 that  $\frac{1}{t} u(x,t) = \underline{f}$  if $\bar x\in \M$.
In fact, for such $\bar x$,
 \begin{equation*}
 u(\overline{x},t)=\inf\limits_{\alpha_{\cdot}} \int_{0}^{t}\frac{1}{2}|\alpha_{s}|^{2}+ f(x(s))\,\text{d}s
   \leq \int_{0}^{t}f(\overline{x})\,\text{d}t = t\uf ,
\end{equation*}
where the 
 inequality follows from the choice $\alpha_{\cdot}\equiv 0$. 
 The other inequality $\geq$ is true for all $x\in\mathds{R}^{n}$ by Lemma \ref{lem: bound average unif x}.

Denote $R\coloneqq \sqrt{4\|f\|_{\infty}}$ and use the gradient bound  \eqref{eq: unif grad est _ proof} to get 
\[
\left|\frac 1t u(x,t)-\uf\right |\leq \frac 1t R \,\text{dist}(x, \M)  \quad \forall \, x\in\R^n, \; t>0 . 
\]
Then
$u(x,t)\to\uf$ 
{locally uniformly}  as $t\to\infty$.


Define now $\varphi_{t}(\cdot):=u(\cdot,t)-\uf t$. We observe that, in view of \eqref{eq: unif grad est _ proof}, 
$|\varphi_{t}(x)|\leq R\, \text{dist}(x, \M)$ and 
$|\varphi_{t}(x) - \varphi_{t}(y)|\leq R|x-y|$. 
Hence, $\{\varphi_{t}(\cdot)\}_{t\geq 0}$ is a locally uniformly bounded and equi-continuous family. 
%
We claim that   $\varphi_{t}(\cdot)\to \psi(\cdot)\in C(\mathds{R}^n)$ locally uniformly as $t\to+\infty$ and
$\psi(\cdot)$ is a 
 viscosity solution of 
\begin{equation}
\label{eq: pde psi}
    \underline{f} + \frac{1}{2}|D\psi(x)|^{2} = f(x),\quad \text{ in } \mathds{R}^{n}.
\end{equation}
To prove the claim define $u_{\eta}(x,t) \coloneqq \varphi_{
{t}/{\eta}} \left(x\right) = u\left(x,\frac{t}{\eta}\right)-\frac{t}{\eta}\underline{f}$. Then we have
\begin{equation*}
    \eta\partial_{t}u_{\eta} + \underline{f} + \frac{1}{2}|Du_{\eta}|^{2} = f(x) ,\quad \text{ in } \mathds{R}^{n}\times (0,\infty).
\end{equation*}
Now consider
 the upper and lower relaxed semilimits 
\begin{equation*}
   \theta(x,t):= \limsup_{\eta \to 0,\, s\to t,\, y\to x}     u_\eta(y,s) ,
   \quad    \zeta(x,t)
:=\liminf_{\eta \to 0,\, s\to t,\, y\to x}     u_\eta(y,s) ,
\end{equation*}
and note that they are finite by the local equiboundedness of $\varphi_t$.
It is well-known from the stability properties of viscosity solutions (see, e.g., \cite{bardi2008optimal}) that they are, respectively, a 
 sub- and supersolution of \eqref{eq: pde psi} for any $t>0$. Moreover, for all $t>0$,
\[
 \theta(x,t)= \limsup_{s\to +\infty,\, y\to x}    \varphi_s(y) = \limsup_{s\to +\infty}    \varphi_s(x) ,
\]
where the last equality comes from the equicontinuity of $\varphi_t$. Similarly, 
$$ \zeta(x,t)=\liminf_{s\to +\infty}    \varphi_s(x) $$
and so both $\theta$ and $\zeta$ do not depend on $t$. Next note that $\varphi_s(x)=0$ for all  $x\in \M$ and it is non-negative everywhere. Then 
$\theta(x) =  \zeta(x) = 0$ on $\de \M$, and they are a sub- and a supersolution bounded from below of \eqref{eq: pde psi} in $\R^n\setminus \M$, where $f(x)-\uf>0$. Then a standard comparison principle for the Dirichlet problem associated to eikonal equations gives
$\theta(x) =  \zeta(x)$. This proves that $\varphi_t$ converges pointwise to $\psi:=\theta=\zeta\geq 0$, and the convergence is locally uniform by the Ascoli-Arzela theorem, which gives the claim. Moreover $\psi$ coincides with the function $v$ found in Theorem \ref{thm: ergodic 2}.
 

Finally, the convergence of the gradient $D_xu(\cdot,t)=D\varphi_t$ to $D\psi$ 
 is a direct consequence of \cite[Theorem  3.3.3]{cannarsa2004semiconcave}, recalling that $|\varphi_{t}(x)|\leq R\, \text{dist}(x, \M)$ and using the uniform semiconcavity estimate in Lemma \ref{lem: semiconcavity}.
\end{proof}

\section{Reaching the minima via optimal control}  %
\subsection{The optimal control problem with target}
\label{sec: prob target}


In this section we consider  the Dirichlet problem
\begin{equation}\label{eq: eikonal - ell}
\left\{\quad
\begin{aligned}
     |\nabla v(x)| &= \ell(x), & x\in \mathds{R}^{n}\setminus \mathfrak{M} ,\\
     v(x) &= 0, & x\in \mathfrak{M} ,
\end{aligned}\right.
\end{equation}
motivated by the ergodic equation \eqref{eq: ergodic pde} of the previous section if $\ell(x)=\sqrt{ 2(f(x) - \underline{f}) }$. 
Here, however, the standing assumptions are only that $\M\subseteq \R^n$ is a closed nonempty set, possibly unbounded,  and 
{\begin{equation}
\label{elle}
\tag{F}
\ell\in C(\R^n) \text{ is bounded }, \; \ell(x)>0 \text{ if } x\in \mathds{R}^{n}\setminus \mathfrak{M} , \;\;  \ell\equiv 0 \text{ on } \M .
\end{equation}
Also define $\overline{\ell}:=\sup\limits_{x\in\mathds{R}^{n}}\ell(x)$.  
The Lipschitz and semiconcavity conditions 
of the previous section (assumptions (B)) will not
be needed 
in most statements of the present section. }

We recall that the continuous viscosity solution of \eqref{eq: eikonal - ell}
is 
 the value function of the control problem
\begin{equation}
\label{eq: control representation v}
    v(x)=\inf\limits_{\alpha
    }\int_{0}^{t_{x}(\alpha)} \ell(y_{x}^{\alpha}(s))\,\text{d}s,
\end{equation}
where $\alpha$ (an admissible control) is a measurable function 
$[0,+\infty) \to B(0,1)$,  the unit ball in $\mathds{R}^{n}$, $t_{x}(\alpha):=\inf\{s \geq 0\,:\, y_{x}^{\alpha}(s) \in \mathfrak{M}\}$, and
\begin{equation}
    \label{eq: dynamics eikonal}
    \dot{y}^{\alpha}_{x}(s) = \alpha(s),\,\forall\,s\geq 0,\quad y_{x}^{\alpha}(0)=x.
\end{equation}


\begin{thm}\label{thm: existence control deter}
Under 
{ assumption \eqref{elle}} there exists an optimal control $\alpha^{*}$ for the problem \eqref{eq: control representation v}.
\end{thm}

\begin{proof}
Notice first that  \eqref{elle} allows to rewrite $v$ as 
\begin{equation*}
    v(x)=\inf\int_{0}^{+\infty} \ell(y_{x}^{\alpha}(s))\,\text{d}s,\; \text{ s.t.:}\; \eqref{eq: dynamics eikonal}\; \text{with}\; s\mapsto\alpha(s)\in B(0,1) \, \text{ measurable}.
\end{equation*}
Fix $x\in\mathds{R}^{n}$ and consider  a minimizing sequence $(y_{k},\alpha_{k})_{k}$, i.e., satisfying
\begin{equation}\label{eq: convergence minimizing seq}
    \lim\limits_{k\to+\infty} \int_{0}^{+\infty} \ell(y_{k}(t))\,\text{d}t = v(x) , \quad  y_{k}(t) = x +\int_{0}^{t}\alpha_{k}(s)\,\text{d}s,\; 
    \forall\, t\geq 0 .
\end{equation}
Fix $N\in\mathds{N}$. 
Using Alaoglu's theorem, we can extract a subsequence that we denote by $(y_{k(N)}, \alpha_{k(N)})$, where $k(N)\to+\infty$, such that
\begin{equation*}
    \begin{aligned}
        & \alpha_{k(N)} \overset{\ast}{\rightharpoonup} \alpha^{\ast}_{N},\; \text{ a.e. in }\, [0,N] ,\\
        & y_{k(N)} \rightarrow y_{N}^{\ast},\; \text{loc. unif. on } \, [0,N] , \\
        \text{and }\; & y_{N}^{\ast}(t) = x + \int_{0}^{t}\alpha_{N}^{\ast}(s)\,\text{d}s,\; \text{for all }\, t\in [0,N].
    \end{aligned}
\end{equation*}
We repeat this procedure in the interval $[0,N+1]$
and  extract from the previous subsequence 
another subsequence $(y_{k(N+1)}, \alpha_{k(N+1)})$ with the same properties  in $[0,N+1]$. Note that
\begin{equation*}
    \begin{aligned}
        & \alpha^{\ast}_{N+1}=\alpha_{N}^{\ast},\; \text{ a.e. in }\, [0,N]\\
        & y^{\ast}_{N+1}=y^{\ast}_{N},\; \text{ in }\, [0,N]
    \end{aligned}
\end{equation*}
This suggests the definition of the candidate optimal pair $(y^{\ast},\alpha^{\ast})$ as
\begin{equation*}
    (y^{\ast},\alpha^{\ast}) \coloneqq (y^{\ast}_{N},\alpha^{\ast}_{N})\quad \text{ in } [0,N].
\end{equation*}
To prove  its optimality 
consider the diagonal subsequence $(y_{N(N)},\alpha_{N(N)})$. 
By the previous construction, for any fixed $T>0$ 
we have 
\begin{equation}\label{eq: convergence subsequence}
    \begin{aligned}
        & \alpha_{N(N)} \overset{\ast}{\rightharpoonup} \alpha^{\ast},\; \text{ a.e. in }\, [0,T],\\
        & y_{N(N)} \rightarrow y^{\ast},\; \text{loc. unif. on } \, [0,T],\\
        \text{and }\; & y^{\ast}(t) = x + \int_{0}^{t}\alpha^{\ast}(s)\,\text{d}s,\; \text{for all }\, t\in [0,T].
    \end{aligned}
\end{equation}
Now use  Fatou's lemma 
\begin{equation*}
    \int_{0}^{\infty} \liminf\limits_{N\to +\infty} \ell(y_{N(N)}(t))\,\text{d}t \leq \liminf\limits_{N\to\infty}\int_{0}^{+\infty} \ell(y_{N(N)}(t))\,\text{d}t.
\end{equation*}
By \eqref{eq: convergence minimizing seq} the right-hand side is 
 $v(x)$ because  $y_{N(N)}$ is a subsequence of $y_{k}$. 
 Now use the continuity of $\ell$ in the left hand side and get
 {
\begin{equation*}
  \int_{0}^{\infty} \ell(y^{\ast}(t))\,\text{d}t  =   \int_{0}^{\infty} \liminf\limits_{N\to +\infty} \ell(y_{N(N)}(t))\,\text{d}t \leq v(x),
\end{equation*}
}
which says that $(y^{\ast},\alpha^{\ast})$ is an optimal pair solution to \eqref{eq: control representation v}.
\end{proof}


Next we show that the fraction of time spent by an optimal trajectory 
 away from the minimizers of $\ell$ 
  tends to zero as $t\to+\infty$. 
  For a given fixed $\delta>0$ we define the set of quasi-minimizers
\begin{equation*}
    K_{\delta}:= \{x\in \mathds{R}^n\,:\, \ell(x)\leq  \delta\}
\end{equation*}
and the fraction of time $\rho^{\delta}(t)$ spent by an optimal trajectory starting from $x$ away from $K_{\delta}$ 
\begin{equation*}
    \rho^{\delta}(t)=  \rho^{\delta}(t,x, \a^*) := \frac{1}{t}\big|\{s\in[0,t]\,:\, y_{x}^{\alpha^{*}}(s)\notin K_{\delta}\}\big| ,
\end{equation*}
where $\big|I\big|$ denotes the Lebesgue measure of $I\subseteq\R$. 
{ In other words, $\rho^{\delta}(t)$ is the image of { the complement of} $K_{\delta}$ by the {\em occupational measure} of the optimal trajectory $ y_{x}^{\alpha^{*}}.$
}
\begin{thm}
\label{thm: conv time average}
{ Under 
assumption \eqref{elle}},   for any $x\in\mathds{R}^{n}$ and $\delta>0$, an optimal trajectory $y_{x}^{\alpha^{*}}(\cdot)$ for the problem \eqref{eq: control representation v} satisfies 
\begin{equation}\label{eq: thm long time average deter}
    \rho^{\delta}(t,x, \a^*) \leq\, \frac{\overline{\ell}}{t\,\delta}\,\text{dist}(x,\mathfrak{M})  .
\end{equation}
In particular, 
$  \lim\limits_{t\to + \infty} \rho^{\delta}(t) = 0$. 
\end{thm}
\begin{proof}
Since $\ell\geq 0$, using the characteristic function $\mathds{1}_{Q}(y)=1$ if $y\in Q$ and 0 otherwise, 
\begin{equation*}
        \int_{0}^{t}\ell(y_{x}^{\alpha^{*}}(s))\text{d}s 
         \geq \int_{0}^{t}\mathds{1}_{K^{c}_{\delta}}(y_{x}^{\alpha^{*}}(s))\,\ell(y_{x}^{\alpha^{*}}(s))\,\text{d}s\;\geq \delta\,\int_{0}^{t}\mathds{1}_{K^{c}_{\delta}}(y_{x}^{\alpha^{*}}(s))\,\text{d}s
\end{equation*}
and hence 
\begin{equation*}
    \frac{1}{t}\int_{0}^{t}\ell(y_{x}^{\alpha^{*}}(s))\text{d}s\; \geq\; \delta\,\rho^{\delta}(t) 
    .
\end{equation*}
Now, since $\ell(y_{x}^{\alpha^{*}}(s)) = 0$ for all $s\geq t_{x}(\alpha^{*})$ { and $\ell(\cdot)\leq \Bar{\ell}$}, we have for all $t\geq 0$
\begin{equation*}
\begin{aligned}
	\int_{0}^{t}\ell(y_{x}^{\alpha^{*}}(s))\,\text{d}s & \;\leq\; \int_{0}^{t_{x}(\alpha^{*})}\ell(y_{x}^{\alpha^{*}}(s))\,\text{d}s,\\
		& \;=\; v(x) \; \leq \; \Bar{\ell}\; \inf 
		\left\{ t_{x}(\alpha) : 
		\eqref{eq: dynamics eikonal} \text{ holds with } |\a(s)|\leq 1\right\}. 
\end{aligned}
\end{equation*}
The second factor on the right-hand side is the minimal time function 
 whose optimal trajectories are the straight lines from the initial position $x$ to its orthogonal projection on the set $\mathfrak{M}$, with maximal speed $1$. Therefore 
  the right-hand side in the last inequality is less or equal $\Bar{\ell}|z-x|$ for any $z\in \mathfrak{M}$, and then 
\begin{equation*}
    v(x) \leq \Bar{\ell} \;\text{dist}(x,\mathfrak{M}).
\end{equation*}
Combining the inequalities we get 
\begin{equation*}
	0 \leq\; \delta\,\rho^{\delta}(t) \leq\; \frac{1}{t}\int_{0}^{t}\ell(y_{x}^{\alpha^{*}}(s))\,\text{d}s \leq\; \frac{v(x)}{t} \leq\; \frac{\Bar{\ell}}{t} \,\text{dist}(x,\mathfrak{M}) ,
\end{equation*}
which concludes the proof.
\end{proof}

\subsection{A 
gradient descent inclusion for the optimal trajectories}
\label{sec: diff inclusion}

So far, we showed that an optimal control exists and the corresponding optimal trajectory does not leave the set of minimizers in average as time goes to infinity, i.e. in the sense of \eqref{eq: thm long time average deter}. We now synthesize optimal feedback controls 
that give 
the gradient descent differential inclusion anticipated in the Introduction. 
{ We recall the definition of subdifferential of a continuous function
\begin{equation*}
	D^{-}v(z)  \coloneqq \left\{ \, p \, :\; \liminf\limits_{x\to z}
	\frac{v(x) - v(z) - p\cdot(x-z)}{|x-z|} \geq 0 \right\}.
\end{equation*}
}

\begin{thm}
\label{thm: characterization optimal trajectory deter}
{ Assume  \eqref{elle}}.
A control $\alpha$ with corresponding trajectory $y(\cdot):=y_{x}^{\alpha}(\cdot)$ is optimal if and only if
\begin{equation}\label{eq: DI}\tag{DI}
    \dot{y}(s)\in \left\{ - \frac{p}{|p|}\,,\; p\in D^{-}v(y(s))\right\},\; \text{ for a.e. }\, s\in \;(0,t_{x}(\alpha)).
\end{equation}
\end{thm}

\begin{proof}
By the dynamic programming principle, the function
\begin{equation}\label{eq: def h proof deter}
    h(t)\coloneqq v(y_{x}^{\alpha}(t)) + \int_{0}^{t} \ell(y_{x}^{\alpha}(s))\text{d}s,\quad 0\leq t\leq t_{x}(\alpha)
\end{equation}
is non-decreasing for all $\alpha$, and non-increasing (hence constant) if and only if $\alpha$ is optimal. And since $h$ is locally Lipschitz, we get
\begin{equation*}
    \alpha\,\text{ is optimal }\, \text{ if and only if }\;\; h'(t)\leq 0\; \text{ a.e. } t.
\end{equation*}
\textit{Proof of Necessity. } Assume $\alpha$ is optimal, and so $h'\leq 0$. Let $y(\cdot):=y_{x}^{\alpha}(\cdot)$.\\
{Claim 1. } $p\cdot \dot{y}(t) + \ell(y(t))\leq 0$ for all $p\in D^{-}v(y(t))$ a.e. $t$.

Let $\partial^{-}v(x;q)$ be the lower Dini derivative at $x$ in the direction $q$ (see equation (2.47) in \cite[page 125]{bardi2008optimal}). Then by \cite[Lemma 2.50, p. 135]{bardi2008optimal}, one has
\begin{equation*}
    \partial^{-}(v\circ y)(s;1) = \partial^{-}v(y(s);\dot{y}(s))
\end{equation*}
and for almost every $t$, $h'(t) = \partial^{-}v(y(t);\dot{y}(t)) + \ell(y(t))$. Next, using \cite[Lemma 2.37, p. 126]{bardi2008optimal}, one has, for any $z\in\mathds{R}^{n}$,
\begin{equation*}
    D^{-}v(z) = \{\,p\,:\;p\cdot q \leq \partial^{-}v(z;q),\;\forall\,q\in\mathds{R}^{n}\},
\end{equation*}
and hence, for almost every $t$ and for all $p\in D^{-}v(y(t))$, 
\begin{equation*}
    p\cdot \dot{y}(t) + \ell(y(t))\leq \partial^{-}v(y(t);\dot{y}(t)) + \ell(y(t)) = h'(t)\leq 0 .
\end{equation*}
{Claim 2. } $\dot{y}(t) = -\frac{p}{|p|}$ for all $p\in D^{-}v(y(t))$, a.e. $t$.

By \cite[Proposition 5.3, p. 344]{bardi2008optimal}, $v$ is a bilateral supersolution of $|Dv(x)|-\ell(x) = 0$ in $\mathds{R}^{n}\setminus\mathfrak{M}$, i.e. $|p|-\ell(x)=0$ for all $p\in D^{-}v(x)$. This implies in particular that $p\neq 0$ if $x\notin \mathfrak{M}$. Hence, and using claim 1 together with $\dot{y}\in B(0,1)$, one gets
\begin{equation*}
    |p|=\ell(y(t)) \leq -p\cdot \dot{y}(t)\leq |p| ,
\end{equation*}
that is, $\dot{y}(t) = -
{p}/{|p|}$.

\textit{Proof of sufficiency. } 
By the non-smooth calculus rule just recalled, for a.e. $t$, 
\begin{equation*}
    h'(t) = - \partial^{-}v(y(t);-\dot{y}(t)) + \ell(y(t)) \leq -p\cdot(-\dot{y}(t)) + \ell(y(t)),\quad \forall\,p\in D^{-}v(y(t)).
\end{equation*}
Then, if we assume $y(\cdot)$ solves \eqref{eq: DI}, 
\begin{equation*}
    h'(t) \leq -p\cdot\frac{p}{|p|} + \ell(y(t)) = -|p|+\ell(y(t)) \leq 0
\end{equation*}
because $v$ is a supersolution of $|Dv|-\ell = 0$ and $p\in D^{-}v(y(t))$. 
\end{proof}

\begin{rmk}
Combining Theorem \ref{thm: existence control deter} and Theorem \ref{thm: characterization optimal trajectory deter}, the differential inclusion \eqref{eq: DI} has at least a solution and all such solutions are optimal.
\end{rmk}

We recall the definition of limiting gradient of a Lipschitz function 
\begin{equation*}
    D^{*}v(z)  \coloneqq \{ \,p\,:\; p=\lim\limits_{n\to +\infty} Dv(x_{n})\;\text{ for some }\, x_{n}\to z \}
\end{equation*}
and the super-differential of a continuous function
\begin{equation*}
	D^{+}v(z)  \coloneqq \left\{ \, p \, :\; \limsup\limits_{x\to z}
	\frac{v(x) - v(z) - p\cdot(x-z)}{|x-z|} \leq 0 \right\}.
\end{equation*}


\begin{thm} 
\label{ode}
{ Assume  \eqref{elle}}. The following necessary and sufficient conditions { of optimality} hold.
\begin{enumerate}[label = (\Roman*)]
    \item 
    { If $y(\cdot)$ is optimal, then }
\begin{enumerate}[label = (\roman*)]
    \item $\dot{y}(t)=-\frac{p}{|p|}$, for all $p\in D^{+}v(y(t)),\, p\neq 0$ and almost all $t\in(0,t_{x}(\a^*))$,
    \item $|p|=\ell(y(t))$, for all $p\in D^{+}v(y(t))$ and all $t\in(0,t_{x}(\a^*))$,
  {  \item $D^{+}v(y(t))$ is a singleton for all $t\in(0,t_{x}(\a^*))$.
\item If $\ell(x)=\sqrt{ 2(f(x) - \underline{f}) }$ and  assumptions (A) and (B) are satisfied,
then $v$ is differentiable at all points 
 $y(t)$ with 
  $t\in(0,t_{x}(\a^*))$ and } 
\begin{equation}
\label{descent}
    \dot{y}(t)= - \frac{Dv(y(t))}{|Dv(y(t))|},\quad \forall \,t\in (0,t_{x}(\a^*)).
\end{equation}
\end{enumerate}

    \item A sufficient condition for the optimality of $y(\cdot)$ is 
\begin{equation}\label{eq: DI 2}
    \dot{y}(t) \in - \left\{ \,\frac{p}{|p|}\;:\;p\in D^{*}v(y(t))\cap D^{+}v(y(t)),\,p\neq 0 \right\} ,\, \text{ a.e. } t.
\end{equation}
\end{enumerate}
\end{thm}

\begin{proof}
To prove {\it (I.i)} we take $h$ defined by \eqref{eq: def h proof deter} and let $\partial^{+}v(x;q)$ be the upper Dini derivative of $v$ in direction $q$, with $|q|=1$.\\
{Claim 1. } $p\cdot\dot{y}(t) + \ell(y(t))\leq 0$,\, for all $p\in D^{*}v(y(t))$, a.e. $t$.

Using \cite[Lemma 2.37, p. 126]{bardi2008optimal}, one has, for any $z\in\mathds{R}^{n}$ 
\begin{equation*}
    D^{+}v(z) = \{\,p\;:\;p\cdot q \geq \partial^{+}v(z;q),\; \forall\,q\in\mathds{R}^{n}  \}.
\end{equation*}
Hence, for $p\in D^{+}v(y(t))$, one has
\begin{equation*}
    p\cdot \dot{y}(t) + \ell(y(t)) = -p\cdot (-\dot{y}(t)) + \ell(y(t)) \leq -\partial^{+}v(y(t);-\dot{y}(t)) + \ell(y(t)).
\end{equation*}
But, as in Claim 1 in the proof of Theorem \ref{thm: characterization optimal trajectory deter}, and since $y$ is optimal, one gets
\begin{equation*}
    -\partial^{+}v(y(t);-\dot{y}(t)) + \ell(y(t)) = h'(t)\leq 0,
\end{equation*}
which proves the claim.\\
{Claim 2. } $\dot{y}(t) = -\frac{p}{|p|}$ for all $p\in D^{+}v(y(t)), p\neq 0$, a.e. $t$.

Recalling $|\dot{y}|\in B(0,1)$ and $v$ being a subsolution of $|Dv|-\ell=0$, we have for all $p\in D^{+}v(y(t))$, $|p|\leq \ell(y(t))\leq -p\cdot \dot{y}(t) \leq |p|$, and hence, either $p=0$ or $\dot{y}(t) = -\frac{p}{|p|}$. 


To prove {\it (I.ii)} we use the fact that $h$ is non-increasing if and only if $y(\cdot)$ is optimal. Hence, for { $t>0$ and}  $\tau>0$ small, one has
\begin{equation*}
{    \begin{aligned}
         h(t) - h(t-\tau) \leq 0 &\Rightarrow v(y(t))-v(y(t-\tau)) + \int_{t-\tau}^{t}\ell(y(s))\text{d}s \leq 0\\
 	& \Rightarrow v(y(t))-v(y(t-\tau)) \leq -\ell(y(t))\tau + o(\tau) .
    \end{aligned}
    }
\end{equation*}
Recalling the definition of $p\in D^{+}v(y(t))$, one has
\begin{equation*}
 { \begin{aligned}
    & v(y(t)) - v(y(t-\tau)) \geq p\cdot (y(t) - y(t-\tau) ) + o(\tau)\\ 
\Rightarrow \; &    v(y(t)) - v(y(t-\tau)) \geq \int_{t-\tau}^{t} p\cdot \alpha(s)\text{d}s  + o(\tau)  
  \geq -|p|\tau + o(\tau) ,
\end{aligned}
}
\end{equation*}
and together with the previous inequality this yields
\begin{equation*}
	|p| \geq \ell(y(t)),\quad \forall\, t\in (0,t_{x}(\alpha^*)).
\end{equation*}
The other inequality is a direct consequence of $p$ being in $D^{+}v(y(t))$ and $v$ a subsolution. This concludes the proof of statement \textit{(I.ii)}.

{ The property \textit{(I.iii)} follows immediately from the equality $|p| = \ell(y(t))$ for all $p\in D^{+}v(y(t))$ and the convexity of the set $D^{+}v(y(t))$.

Under the additional conditions of   \textit{(I.iv)}, $v$ is semiconcave thanks to Lemma \ref{lem: semiconcavity - lambda} (or Lemma \ref{lem: semiconcavity}). 
This implies that $v$ is differentiable at all points where the superdifferential is a singleton 
(see, e.g.,  \cite[Proposition II.4.7 (c), p. 66]{bardi2008optimal}), and then at all $y(t)$ with $t\in (0,t_{x}(\alpha^{*}))$. Hence, \eqref{eq: DI} becomes \eqref{descent}.
}

To prove {\it (II)} note that at all points of differentiability of $v$, one has $|Dv(z)| = \ell(z)$. Then for all $p\in D^{*}v(z)$, $|p|=\ell(z)$. And one has
\begin{equation*}
    h'(t) = \partial^{+}v(y(t);\dot{y}(t)) + \ell(y(t)) \leq p\cdot \dot{y}(t) + \ell(y(t)),\quad \forall\,p\in D^{+}v(y(t)).
\end{equation*}
Then, for $y$ solving \eqref{eq: DI 2}, $p\neq 0$
\begin{equation*}
    h'(t)\leq -p\cdot\frac{p}{|p|} + \ell(y(t)) = -|p|+\ell(y(t)) = 0
\end{equation*}
which concludes the proof as it has been done for Theorem \ref{thm: characterization optimal trajectory deter}.
\end{proof}


\subsection{Convergence of optimal trajectories to the argmin} 
\label{sec: conv traj}

In order to show stability of $\mathfrak{M}$, we need an assumption which prevents $\ell(\cdot)$ from approaching $0$ when $\text{dist}(x,\mathfrak{M})\to \infty$,  
that is,
\begin{itemize}
    \item for all $\delta>0$, there exists  $\gamma=\gamma(\delta)>0$ such that
        \begin{equation}
        \label{eq: assumption for stab}
        \tag{H}
             \inf\{\ell(x)\,:\,\text{dist}(x,\mathfrak{M})\,> \delta\} \,>\, \gamma(\delta).
        \end{equation}
\end{itemize}
If $\M$ is bounded, then it is easy to see that this condition is equivalent to
\[
\liminf_{|x|\to\infty} \ell(x)>0 ,
\]
which is also equivalent to { Assumption (A3) in \cite{ishii2020vanishing}, Assumption (L3)-(3.2) in \cite{cannarsa2022asymptotic}, and } Assumption (L3) in \cite
{cannarsa2020asymptotic}. The last inequality, however, is impossible when $\M$ is unbounded.

\begin{rmk}
An example of function 
with  a unique global minimizer that does not satisfy hypothesis \eqref{eq: assumption for stab} is 
$\ell(x)=|x|e^{-x^2}.$ In this case  $\mathfrak{M}=\{0\}$ and $\inf\{
\ell(x) : 
|x|>\d\} = 0$ for all $\d$.
\end{rmk}

A direct consequence of Theorem \ref{thm: characterization optimal trajectory deter} is the following result.

\begin{corollary}\label{cor: lower bound of rho}
Assume the conditions { \eqref{elle}} and  \eqref{eq: assumption for stab}. 
Let $y_{x}^{\alpha^{*}}(\cdot)$ be an optimal trajectory 
and $\delta>0$. If there exists $\tau>0$ such that $\text{dist}(y^{*}(\tau),\mathfrak{M})>\delta$, then, for $\gamma(\cdot)$ defined in \eqref{eq: assumption for stab},
\begin{equation}
    \rho^{\gamma(\delta/2)}(t) \geq \frac{\delta}{t},\quad\quad \forall\,t>\tau + \frac{\delta}{2} .
\end{equation}
\end{corollary}

\begin{proof}
Set $y^{*}(\cdot)\coloneqq y_{x}^{\alpha^{*}}(\cdot)$.  Since it 
 satisfies \eqref{eq: DI}, we have $|\dot{y}^{*}(\cdot)|\leq 1$ and hence $y^{*}(\cdot)$ is Lipschitz continuous. Therefore, given $\delta>0$, if there exists $\tau>0$ such that $\text{dist}(y^{*}(\tau),\mathfrak{M})>\delta$, then
\begin{equation*}
    \begin{aligned}
        \delta < \text{dist}(y^{*}(\tau),\mathfrak{M}) & \leq \text{dist}(y^{*}(s),\mathfrak{M}) + |y^{*}(s)-y^{*}(\tau)|\\
        & \leq \text{dist}(y^{*}(s),\mathfrak{M}) + |s-\tau|\\
    \end{aligned}
\end{equation*}
which yields
\begin{equation*}
    \text{dist}(y^{*}(s),\mathfrak{M})\,>\,\frac{\delta}{2},\quad \forall\,s\in ]\tau-\delta/2,\tau+\delta/2[.
\end{equation*}
Hence one has
\begin{equation*}
    \ell(y^{*}(s))\geq \inf\left\{\ell(x)\,:\,\text{dist}(x,\mathfrak{M})\,> \frac{\delta}{2}\right\},\quad \forall\,s\in ]\tau-\delta/2,\tau+\delta/2[ ,
\end{equation*}
and together with \eqref{eq: assumption for stab}, one gets
\begin{equation}\label{eq: lower bound of l in proof}
    \ell(y^{*}(s))> \gamma(\delta/2),\quad \forall\,s\in ]\tau-\delta/2,\tau+\delta/2[ .
\end{equation}
Therefore 
\begin{equation*}
    |\{s\in [0,t]\,:\, y^{*}(s) \notin K_{\gamma(\delta/2)}\}| \geq |\,]\tau-\delta/2,\tau+\delta/2[\,|,\quad \forall\,t>\tau +\frac{\delta}{2}.
\end{equation*}
The latter writes as
\begin{equation*}
    t\,\rho^{\gamma(\delta/2)}(t) \geq \delta
\end{equation*}
and concludes the proof.
\end{proof}

We are now ready to show stability properties of the set of global minimizers $\mathfrak{M}$ with respect to the optimal trajectories $y_{x}^{\alpha^{*}}(\cdot)$. 

\begin{thm}
\label{thm: stability of M}
Assume { \eqref{elle} and \eqref{eq: assumption for stab} hold.} 
Then for $y^{*}(\cdot)$ as in \eqref{eq: DI}, 
\begin{enumerate}[label = (\roman*)]
    \item $\mathfrak{M}$ is Lyapunov stable
    \footnote{This means that 
    $\forall\,\varepsilon>0,\;\exists\,\eta>0$ such that $\text{dist}(x,\mathfrak{M})\leq \eta \Rightarrow \text{dist}\left(y_{x}^{\alpha^{*}}(t),\mathfrak{M}\right)\leq \varepsilon$, $\forall\,t\geq0$.},
    \item $\mathfrak{M}$ is globally asymptotically stable
    \footnote{This means that $\M$ is Lyapunov stable and $\lim\limits_{t\to+\infty}\text{dist}\left(y_{x}^{\alpha^{*}}(t),\mathfrak{M}\right) = 0$ for all $x\in\mathds{R}^{n}$.}.
\end{enumerate}
\end{thm}

\begin{proof}
Let $y^{*}(\cdot)\coloneqq y_{x}^{\alpha^{*}}(\cdot)$ be an optimal trajectory, i.e.,  
a solution of  \eqref{eq: DI}. We proceed by contradiction. 

\textit{Proof of (i).} Let $\varepsilon>0$ be fixed and suppose for all $\eta>0$, $\exists\,\tau>0$ such that $\text{dist}(y^{*}(\tau),\mathfrak{M})>\varepsilon$ and $\text{dist}(x,\mathfrak{M})<\eta$. Then from Corollary \ref{cor: lower bound of rho}, one has
\begin{equation*}
    \rho^{\gamma(\varepsilon/2)}(t) \geq \frac{\varepsilon}{t},\quad \forall\, t > \tau+\frac{\varepsilon}{2}.
\end{equation*}
And from Theorem \ref{thm: conv time average}, one has
\begin{equation*}
    \frac{t\,\gamma(\varepsilon/2)}{\overline{\ell}} \rho^{\gamma(\varepsilon/2)}(t) \leq \text{dist}(x,\mathfrak{M}).
\end{equation*}
Therefore one gets
\begin{equation*}
    \frac{\varepsilon\,\gamma(\varepsilon/2)}{\overline{\ell}} \leq \text{dist}(x,\mathfrak{M})
\end{equation*}
which contradicts $\text{dist}(x,\mathfrak{M})<\eta$ when we choose $\eta< \frac{\varepsilon\,\gamma(\varepsilon/2)}{\overline{\ell}}$. Hence we can conclude that, for all $\varepsilon>0$, there exists $\eta>0$ such that if $\text{dist}(x,\mathfrak{M})\leq \eta$ then $\text{dist}(y^{*}(t),\mathfrak{M})\leq \varepsilon$ for all $t$. 

\textit{Proof of (ii).} Suppose there exists a diverging sequence $\{\tau_{k}\}_{k\geq0}$ and $\varepsilon>0$ such that $\text{dist}(y^{*}(\tau_{k}),\mathfrak{M})>\varepsilon$. Without loss of generality, one can extract a subsequence (again denoted by $\tau_{k}$) such that $\tau_{k+1}-\tau_{k}\geq \varepsilon$. Using Corollary \ref{cor: lower bound of rho}, in particular \eqref{eq: lower bound of l in proof}, one has for all $k\geq 0$
\begin{equation*}
    \ell(y^{*}(s))\geq \gamma(\varepsilon/2),\quad \forall\,s\in ]\tau_{k}-\varepsilon/2,\tau_{k}+\varepsilon/2[
\end{equation*}
and therefore
\begin{equation*}
    |\{s\in [0,t]\,:\, y^{*}(s) \notin K_{\gamma(\varepsilon/2)}\}|\, > \sum\limits_{\{k \geq 0\,:\,\tau_{k} \leq t-\frac{\varepsilon}{2}\}} |\,]\tau_{k}-\varepsilon/2,\tau_{k}+\varepsilon/2[\,|
    = N(t)\,\varepsilon,
\end{equation*}
where $N(t)$ is the number of distinct elements $\{\tau_{k}\}_{k\geq0}$ that are in $[0,t+\varepsilon/2]$, i.e. 
\begin{equation*}
    N(t) \coloneqq \#\{\tau_{k}\,:\, \tau_{k}\leq t+\varepsilon/2,\;k\geq 0\}.
\end{equation*}
The previous inequality writes as
\begin{equation*}
    t\rho^{\gamma(\varepsilon/2)}(t)\,> N(t)\,\varepsilon.
\end{equation*}
On the other hand, we know from Theorem \ref{thm: conv time average}, in particular \eqref{eq: thm long time average deter}, that
\begin{equation*}
    t\rho^{\gamma(\varepsilon/2)}(t) \leq \frac{\overline{\ell}\,\text{dist}(x,\mathfrak{M})}{\gamma(\varepsilon/2)},
\end{equation*}
and so we have $N(t) < \frac{\overline{\ell}\,\text{dist}(x,\mathfrak{M})}{\varepsilon\,\gamma(\varepsilon/2)}$. But this cannot be true since $N(t)\to +\infty$ as $t\to+\infty$, and hence it concludes the proof.
\end{proof}

\subsection{On reaching the argmin 
 in finite time}
 \label{sec: finite time}

{ Here we investigate whether the hitting time $t_{x}(\alpha^{*})$ of an optimal trajectory with the target $\M$ is finite or not. In view of the gradient descent inclusion \eqref{grad_desc}, or its smooth version \eqref{descent}, the question is equivalent to the finite length of the orbits of the gradient flow $\dot y \in - D^- v(y)$, or $\dot y = -\nabla v(y)$. This is a classical problem with a large literature. Positive results require strong regularity of $v$, such as quasiconvexity and subanaliticity \cite{bolte2007lojasiewicz}. On the other hand, counterexamples are known for $v\in C^\infty(\R^2)$ and target a circle \cite{palis2012geometric} or a single point \cite{daniilidis2010asymptotic}.

In our case $v$ is not smooth, but it is the value function of a control problem and solves an eikonal equation. These properties can be exploited
to prove  that the hitting time is finite in some cases. 
}

{ 
The first sufficient condition, that complements the hypothesis \eqref{eq: assumption for stab}, is the following, where 
    $d(x)\coloneqq \text{dist}(x,\mathfrak{M})$: }
\begin{itemize}
    \item there exists a continuous function $
    \tilde{\gamma}(s)>0$ for all $s>0$ and $\tilde{\gamma}(0)=0$, and some  $r>0$ such that
    \begin{equation}
    \label{eq: assumption L finite time}
    \tag{L}
    \ell(x)=\tilde{\gamma}(d(x)),\quad \forall\,x \,\text{ s.t. }\, d(x)\leq r .
    \end{equation}
\end{itemize}


\begin{prop}
\label{prop}
Assume { \eqref{elle}, \eqref{eq: assumption for stab}, }and \eqref{eq: assumption L finite time} hold, 
and $\a^*$ be an optimal control for problem \eqref {eq: control representation v}. 
Then the hitting time $t_{x}(\alpha^{*})=d(x)$ whenever $d(x)\leq r$ and it is finite for all $x$.
\end{prop}

\begin{proof}
Let us first note that the finiteness for all $x$ follows from the property in the case $d(x)\leq r$,
because by Theorem \ref{thm: stability of M} \textit{(ii)} there exists a finite time $\widetilde{t}_{x}$ such that $d(y_{x}^{\alpha^{*}}(\widetilde{t}_{x})) \leq r$. 

We 
 assume 
 that the initial position $x$ satisfies $d(x)\leq r$
and aim  to prove that
\begin{equation}
    \label{eq: value function - proof finite time}
     v(x) = \int_{0}^{d(x)} \tilde{\gamma}(s)\,\text{d}s ,
\end{equation}
where  $v(x)$ is the value function defined in \eqref{eq: control representation v}.
Denote by $V(x)$ the right-hand side of the last equality. 

We first claim that $v(x)\leq V(x)$. Take $z$ is in the set of projections of $x$ onto $\mathfrak{M}$ and 
consider the straight line from $x$ to $z$ 
given by the trajectory $\overline{y}_{x}(t) = x - pt$, $t\geq 0$, where $p =\frac{x-z}{|x-z|}$.
Note that $\overline{t}_{x}\coloneqq \inf\{t\geq 0\,:\; \overline{y}_{x}(s)\in \mathfrak{M}\} 
 = d(x)$, and 
  that $d(x-pt)\leq r$ for all $0\leq t\leq \overline{t}_{x}$. Then,  by \eqref{eq: assumption L finite time},
\begin{equation*}
    v(x) \leq \int_{0}^{\overline{t}_{x}}\ell(\overline{y}_{x}(t))\,\text{d}t = \int_{0}^{\overline{t}_{x}}\tilde{\gamma}(d(\overline{y}_{x}(t)))\,\text{d}t =: J(x).
\end{equation*}
Observe now that $d(\overline{y}_{x}(t)) = \big| |x-z| - t \big|=d(x) - t$. 
 Therefore, using the change of variable $s\coloneqq d(\overline{y}_{x}(t))=d(x)-t$, we obtain
\begin{equation*}
    J(x) =  \int_{0}^{d(x)}\tilde{\gamma}(d(\overline{y}_{x}(t)))\,\text{d}t = \int_{0}^{d(x)}\tilde{\gamma}(s)\,\text{d}s = V(x)
\end{equation*}
and this proves the claim. 

Next we show that $v(x)\geq V(x)$. 
Since $v(x)$ is a continuous viscosity solution to \eqref{eq: eikonal - ell}, then using \cite[Theorem 3.2 (ii)]{soravia1999optimality} it satisfies the upper optimality principle \cite[Definition 3.1]{soravia1999optimality}, that is,
\begin{equation*}
    v(x) \geq \inf\limits_{\alpha}
     \int_{0}^{t}\ell(y_{x}^{\alpha}(s))\,\text{d}s + v(y_{x}^{\alpha}(t)),\quad \forall t\geq 0,
\end{equation*}
where the dynamics of $y_{x}^{\alpha}(\cdot)$ is again \eqref{eq: dynamics eikonal} with $|\a (s)|\leq 1$. Using 
 \eqref{eq: assumption L finite time} and  
  $v\geq 0$ we get 
\begin{equation*}
    v(x) \geq \inf\limits_{\alpha}
     \int_{0}^{t}\tilde{\gamma}(d(y_{x}^{\alpha}(s)))\,\text{d}s,\quad \forall t\geq 0.
\end{equation*}
In particular, 
since $\tilde{\gamma}(s) = 0$ if and only if $s=0$, 
we have
\begin{equation*}
    v(x) \geq \inf\limits_{\alpha \in B(0,1)} \int_{0}^{t_{x}(\alpha)}\tilde{\gamma}(d(y_{x}^{\alpha}(s)))\,\text{d}s\;=:\,W(x) .
\end{equation*}
Then the  function $W(x)$ solves in the viscosity sense 
the Dirichlet problem
\begin{equation}
\left\{\quad
\begin{aligned}
     |\nabla W(x)| &= \tilde{\gamma}(d(x)), & x\in \mathds{R}^{n}\setminus \mathfrak{M}\\
     W(x) &= 0, & x\in \mathfrak{M}.
\end{aligned}\right.
\end{equation}
But $V(x)\coloneqq\int_{0}^{d(x)} \tilde{\gamma}(s)\,\text{d}s$ is also a viscosity solution of this Dirichlet problem because $|D^{\pm}V(x)|=|D^{\pm}d(x)|\tilde{\gamma}(d(x))$.
We conclude using 
\cite[Theorem 1 and Remark 3.1]{malisoff2004bounded} that $V(x)=W(x)$ and hence $v(x)\geq V(x)$.

Finally we use in the integral of the formula \eqref{eq: value function - proof finite time} the same change of variable as 
above to get
\begin{equation*}
    v(x) = \int_{0}^{d(x)}\tilde{\gamma}(d(\overline{y}_{x}(t)))\,\text{d}t = \int_{0}^{d(x)}\ell(\overline{y}_{x}(t))\,\text{d}t .
\end{equation*}
This proves that $\overline{y}_{x}(t)\coloneqq x-pt$ is an optimal trajectory  and $d(x)$ is its hitting time. 
\end{proof}
{ \begin{rmk}
In some control problems it may happen  that an optimal trajectory remains arbitrarily close to a target 
 without ever reaching it. Such a behavior has been observed in 
a linear-quadratic control problem studied in \cite[\S 6.1]{kouhkouh2018dynamic} with the target 
is a singleton $\{x_{\circ}\}$ and the time $t_{\varepsilon}$ of being $\varepsilon$-close to $x_{\circ}$ is shown to be 
$t_{\varepsilon}=C\,\ln\left(\frac{|x-x_{\circ}|}{\varepsilon}\right)$, where $x$ is  the initial state.
Moreover, 
 an optimal trajectory oscillates periodically around $x_{\circ}$ (see \cite[p. 55]{kouhkouh2018dynamic}).
\end{rmk}
}


{ 
Next we show that, under the set of assumptions of Section \ref{sec: asymptotics}, a bound from below on $\ell$ near the target is a sufficient condition for the finite hitting time. The proof uses an inequality of  \L ojasiewicz type along optimal gradient orbits.

\begin{thm} 
Assume $\ell(x)=\sqrt{ 2(f(x) - \underline{f}) }$, 
 (A), (B), and (H) are satisfied, and for some $c, r>0$, $0<\beta<3/2$,
  \begin{equation}
 \label{lbelow}
  \ell(x) \geq c\, d(x)^\beta ,\quad \forall\,x \,\text{ s.t. }\, d(x)\leq r .
 \end{equation}
 If $\a^*$ is   an optimal control  for  $x$, then the hitting time $t_{x}(\alpha^{*})$ 
  is finite for all $x$, and for $d(x)$ sufficiently small 
   \begin{equation}
 \label{est-t}
 t_{x}(\alpha^{*})\leq     \frac C{1- 2\beta/3}\, d(x)^{\frac 32 - \beta}  .
   \end{equation} 
\end{thm} 
\begin{proof}
Set $y(t):= y_x^{\alpha^{*}}(t)$ and recall from Theorem \ref{thm: stability of M} that
 \begin{equation*}
\lim_{t\to t_{x}(\alpha^{*})} d(y(t)) = 0.
 \end{equation*}
Therefore it is not restrictive to assume that $d(y(t))\leq r$ for all $t>0$.

We re-parametrise the trajectory $y$ to get a gradient orbit.
Set  
$$
s(t):=\int_0^t|Dv(y(\tau))|^{-1} d\tau \in [0, T) , \quad 0\leq t < t_{x}(\alpha^{*}),
$$
 where $T\leq +\infty$. Define $s\mapsto t(s)$, $[0, T)\to [0, t_{x}(\alpha^{*}))$,   the inverse  function of $s(t)$ and $z(s):= y(t(s))$. Then
\[
\dot z(s) = - Dv(z(s)) , \quad z(0)=x , \quad \lim_{s\to T}  d(z(s)) = 0 ,
\]
and 
\[
t(s)=\int_0^s |Dv(z(\tau))| d\tau =  \int_0^s |\dot z(\tau)|d\tau .
\]
Therefore 
\[
t_{x}(\alpha^{*})=\lim_{s\to T} t(s) =  \int_0^{T} |\dot z(\tau)|d\tau ,
\]
and so    $t_{x}(\alpha^{*})<\infty$ if the length of the gradient orbit $z(\cdot)$ is finite. By Theorem \ref{ode}, $v$ is differentiable at all points $z(s)$, $s>0$, and then
 \begin{equation}
 \label{half L}
|Dv(z(s))|= \ell(z(s)) \geq c d(z(s))^\beta , \quad \forall s>0 
 \end{equation}
by \eqref{lbelow} and $d(z(s))\leq r$. On the other hand, by assumptions (A2) and (B1), for some $C_3>0$
\[
\ell(x)\leq C_3 \sqrt{d(x)} .
\]
 By repeating the 1st half of the proof of Proposition \ref{prop} we get
   \begin{equation}
 \label{est-v}
 v(x)\leq  \int_{0}^{d(x)} C_3 \sqrt s\,\text{d}s = \frac {2C_3}{3} d(x)^{3/2} .
 \end{equation}
By combining this with \eqref{half L} we obtain
\[
|Dv(z(s))| \geq C_4 v(z(s))^\rho ,
\]
where $\rho:= 2\beta/3 <1$. This is a \L ojasiewicz inequality along the gradient orbit $z(\cdot)$, and we can use the following classical argument:
\[
\frac{-1}{1-\rho} \frac d{ds}[v(z(s))^{1-\rho}] = \frac {-Dv(z(s))\cdot \dot z(s)}{v(z(s))^{\rho}} = \frac {|Dv(z(s))| |\dot z(s)|}{v(z(s))^{\rho}} \geq C_4 |\dot z(s)| ,
\]
which integrated from $0$ to $T$ gives 
\[
t_{x}(\alpha^{*}) \leq \frac{v(x)^{1-\rho}}{C_4(1-\rho)} .
\]
Now we combine this with \eqref {est-v} to get the estimate   \eqref {est-t}.
\end{proof}

}

{\subsection*{Acknowledgement} The authors wish to thank Piermarco Cannarsa and Olivier Ley for useful conversations and the referees for their careful reading and insightful remarks.
}

\bibliography{bibliography}
\bibliographystyle{siam}
\end{document}